\documentclass[10pt]{article}

\usepackage{amsmath,amsbsy,amssymb,latexsym,amsthm,dsfont}
\usepackage{a4wide}

\newtheorem{Theorem}{Theorem}

\newtheorem{Definition}{Definition}
\newtheorem{Remark}{Remark}

\begin{document}

\title{A Scale Variational Principle of Herglotz}

\author{Ricardo Almeida\\
\texttt{ricardo.almeida@ua.pt}}

\date{Center for Research and Development in Mathematics and Applications (CIDMA)\\
Department of Mathematics, University of Aveiro, 3810--193 Aveiro, Portugal}

\maketitle

\begin{abstract}

The Herglotz problem is a generalization of the fundamental problem of the calculus of variations. In this paper, we consider a class of non-differentiable functions, where the dynamics is described by a scale derivative. Necessary conditions are derived to determine the optimal solution for the problem. Some other problems are considered, like transversality conditions, the multi-dimensional case, higher-order derivatives and for several independent variables.

\bigskip


\smallskip

\end{abstract}


\section{Introduction}

The calculus of variations deals with optimization of a given functional, whose algebraic expression is the integral of a given function, that depends on time, space and the velocity of the trajectory:
$$x \mapsto  \int_a^b L(t,x(t),\dot x(t))\, dt.$$
 The variational principle of Herglotz can be seen as an extension of such classical theories, but instead of an integral, we have the functional as a solution of a differential equation (see \cite{Guenther2,Herglotz}):
$$\left\{\begin{array}{l}
\dot{z}(t)=L(t,x(t),\dot x(t),z(t)), \quad \mbox{ with } \, t\in[a,b],\\
z(a)=z_a .\end{array}\right.$$
Without the dependence of $z$, we can convert this problem into a calculus of variations problem. In fact, integrating the differential equation
$$\dot{z}(t)=L(t,x(t),\dot x(t))$$
from $a$ to $b$, we obtain
$$z(b)=\int_a^b\left[L(t,x(t),\dot x(t))+\frac{z_a}{b-a}\right]\,dt.$$
Recently, more advances were made namely proving Noether's type theorems for the variational principle of Herglotz (see e.g. \cite{Georgieva1,Georgieva2,Georgieva3,Guenther1,Guenther2,Orum}). The aim of this paper is to consider the Herglotz problem, but the trajectories $x(\cdot)$ may be non-differentiable functions. We believe that this situation may model more efficiently  certain physical problems, like fractals.

The organization of the paper is the following. In Section \ref{sec2} we define what is a scale derivative, following the concept as presented in \cite{Cresson1}, and we present some of its main properties, like the algebraic rules, integration by parts formula, etc. In Section \ref{sec3} we prove our new results. After presenting the Herglotz scale problem, we prove a necessary condition that every extremizer must fulfill. Some generalizations of the main result are also presented to complete the study.

\section{Scale calculus}
\label{sec2}

We review some definitions and the main results from \cite{Cresson1} that we will need. For more on the subject, see references \cite{Almeida1,Cresson1,Cresson2}.

From now on, let $\alpha, \beta,h$ be reals in $]0,1[$ with $\alpha+\beta>1$ and $h \ll 1$, and consider $I:=[a-h,b+h]$.

\begin{Definition} Let $f:I\rightarrow \mathbb{R}$ be a function. The delta derivative of $f$ at $t$ is  defined by
$$\Delta_h[f](t):=\frac{f(t+ h)-f(t)}{h}, \quad \mbox{for} \quad t\in[a-h,b],$$
and the nabla derivative of $f$ at $t$ is  defined by
$$\nabla_h[f](t):=\frac{f(t)-f(t-h)}{h}, \quad \mbox{for} \quad t\in[a,b+h].$$
\end{Definition}

If $f$ is differentiable, then
$$\lim_{h\to 0}\Delta_{h}[f](t)=\lim_{h\to 0}\nabla_{h}[f](t) = f'(t).$$

These two operators can be combined into a single one, where the real part is the mean value of such operators, and the complex part measures the difference between them.

\begin{Definition}
The $h$-scale derivative of $f$ at $t$ is given by
\begin{equation}
\label{eq:def:cp}
\frac{{\Box_{h}}f}{\Box t}(t)
=\frac12 \left[ \left( \Delta_{h}[f](t) + \nabla_{h}[f](t) \right)+i \left( \Delta_{h}[f](t) - \nabla_{h}[f](t) \right) \right], \quad \mbox{for}\quad  t \in [a,b].
\end{equation}
\end{Definition}

For complex valued functions $f$, such definition is extended by
$$\frac{{\Box_{h}}f}{\Box t}(t)= \frac{{\Box_{h}}\mbox{Re}f}{\Box t}(t)+ i \frac{{\Box_{h}}\mbox{Im}f}{\Box t}(t).$$
We now explain how to drop the dependence on the parameter $h$ in the definition of the scale derivative. First, consider the set ${C^0_{conv}}([a,b]\times ]0,1[,\mathbb{C})$ of the functions  $g\in {C^0}([a,b]\times ]0,1[,\mathbb{C})$ for which the limit
$$\lim_{h\to 0}g(t,h)$$
exists for all $t\in [a,b]$, and let $E$ be a complementary space of ${C^0_{conv}}([a,b]\times ]0,1[,\mathbb{C})$ in ${C^0}([a,b]\times  ]0,1[,\mathbb{C})$.

Define  $\pi$ the projection of ${C^0_{conv}}([a,b]\times ]0,1[,\mathbb{C})\oplus E $ onto ${C^0_{conv}}([a,b]\times ]0,1[,\mathbb{C})$,
$$\begin{array}{lcll}
\pi: & {C^0_{conv}}([a,b]\times ]0,1[,\mathbb{C})\oplus E & \to & {C^0_{conv}}([a,b]\times ]0,1[,\mathbb{C})\\
&g:= g_{conv}+g_E  & \mapsto & \pi(g)=g_{conv}.
\end{array}$$

Using these definitions, we arrive at the main concept of \cite{Cresson1}.

\begin{Definition}
\label{def:ourHD}
The scale derivative of $f\in {C^0}(I,\mathbb{C})$, denoted by $\frac{\Box f}{\Box t}$,  is defined by
\begin{equation}
\label{eq:scaleDer}
\frac{\Box f}{\Box t}(t):=\left< \frac{{\Box_{h}}f}{\Box t} \right>(t), \quad t \in [a,b],
\end{equation}
where
$$\left< \frac{{\Box_{h}}f}{\Box t} \right>(t):=   \lim_{h\to 0}\pi\left(\frac{{\Box_{h}}f}{\Box t}(t)\right).$$
\end{Definition}

\begin{Definition}
Given $f: I^n= [a-nh,b+nh]\rightarrow \mathbb{C}$, define the higher-order scale derivative of $f$ by
$$\frac{\Box^n f}{\Box t^n}(t)=\frac{\Box}{\Box t}\left( \frac{\Box^{n-1} f}{\Box t^{n-1}} \right)(t), \quad t\in[a,b],$$
where $\frac{\Box f^1}{\Box t^1}:= \frac{\Box f}{\Box t}$ and $\frac{\Box f^0}{\Box t^0}:=f$.
\end{Definition}

We will adopt the notation $\Box^n f(t)$ instead of $\frac{\Box^n f}{\Box t^n}(t)$ when there is no danger of confusion.

Scale partial derivatives are also considered here. They are defined as in the standard case.

\begin{Definition}
Let $f:\prod_{i=1}^n[a_i-h,b_i+h]\to\mathbb{R}$ be a function. Define, for each $i\in\{1,\ldots,n\}$,
$$\Delta_h^i[f](t_1,\ldots,t_n):=\frac{f(t_1,\ldots,t_{i-1},t_i+ h,t_{i+1},\ldots,t_n)-f(t_1,\ldots,t_{i-1},t_i,t_{i+1},\ldots,t_n)}{h},$$
for $t_i\in[a_i-h,b_i]$ and for $t_j\in[a_j-h,b_j+h] \, \mbox{ if } \, j\not=i$,
and
$$\nabla_h^i[f](t_1,\ldots,t_n):=\frac{f(t_1,\ldots,t_{i-1},t_i,t_{i+1},\ldots,t_n)-f(t_1,\ldots,t_{i-1},t_i-h,t_{i+1},\ldots,t_n)}{h},$$
for $t_i\in[a_i,b_i+h]$ and for $t_j\in[a_j-h,b_j+h], \, \mbox{ if } \, j\not=i.$
The $h$-scale partial derivative of $f$ with respect to the $i-th$ coordinate is given by
$$\frac{{\Box_{h}}f}{\Box t_i}(t_1,\ldots,t_n)
=\frac12 \left[ \left( \Delta_{h}^i[f]+ \nabla_{h}^i[f] \right)+i \left( \Delta_{h}^i[f]- \nabla_{h}^i[f]\right) \right],$$
for  $t_i \in [a_i,b_i].$
\end{Definition}

The definition of partial scale derivatives $\Box f/\Box t_i$ is clear.

In what follows, we will denote
$$C^n_\Box([a,b], \mathbb{K}):= \{f \in C^0(I^n, \mathbb{K})\mid \frac{\Box^k f}{\Box t^k}\in C^0(I^{n-k}, \mathbb{C}), k=1,2,\ldots,n \}, \quad  \mathbb{K}=
 \mathbb{R} \mbox{ or }  \mathbb{K}= \mathbb{C}.$$

\begin{Definition}
Let $f\in C^0(I,\mathbb{C})$ and $\alpha\in]0,1[$ . We say that $f$ is H\"{o}lderian of H\"{o}lder exponent $\alpha$ if there exists a constant $C>0$ such that,
for all $s,t \in I$,
$$|f(t)-f(s)|\leq C |t-s|^\alpha,$$
and we write $f \in H^\alpha(I, \mathbb{C})$, or simply $f\in H^\alpha$ when there is no danger of mislead.
\end{Definition}

We say that $f(t_1,\ldots,t_n)\in H^\alpha$ if $f(t_1,\ldots,t_{i-1},\cdot,t_{i+1},\ldots,t_n)\in H^\alpha$, for all $i\in\{1,\ldots,n\}$ and for all  $t_j\in[a_j,b_j], \, \, j\not=i$.

\begin{Theorem}\label{LeibnizRule}
For all $f\in H^\alpha$ and $g\in H^\beta$, we have
\begin{equation*}\frac{\Box (f.g)}{\Box t}(t)=\frac{\Box f}{\Box t}(t).g(t)+f(t).\frac{\Box g}{\Box t}(t), \quad t \in [a,b].
\end{equation*}
\end{Theorem}

\begin{Theorem}
\label{Barrow}
Let $f\in C^1_{\Box}([a,b],\mathbb{R})$ be such that
\begin{equation}
\label{nec_condition}
\lim_{h\to 0} \int_a^b \left(\frac{\Box_h f}{\Box t}\right)_E(t) \, dt=0,
\end{equation}
where
$ \frac{\Box_h f}{\Box t}:=  \left(\frac{\Box_h f}{\Box t}\right)_{conv} +  \left(\frac{\Box_h f}{\Box t}\right)_E.$
Then,
$$\int_a^b \frac{\Box f}{\Box t}(t)\, dt=f(b)-f(a).$$
\end{Theorem}

As a consequence, we have the following integration by parts formula. If
$$\lim_{h\to 0} \int_a^b \left(\frac{\Box_h (f \cdot g)}{\Box t}\right)_E(t) \, dt=0,$$
where $f\in H^\alpha$ and  $g \in H^\beta$,then
$$\int_a^b \frac{\Box f}{\Box t}(t) \cdot g(t) dt = \left[f(t)g(t)\right]_a^b - \int_a^b f(t)\cdot \frac{\Box g}{\Box t}(t) dt.$$


\section{The scale variational principle of Herglotz}
\label{sec3}

The (classical) variational principle of Herglotz is described in the following way. Consider the differential equation
$$\left\{ \begin{array}{l}
\dot{z}(t)=L(t,x(t),\dot x(t),z(t)), \quad \mbox{ with } \, t\in[a,b]\\
z(a)=z_a\\
x(a)=x_a, \, x(b)=x_b,
\end{array}\right.$$
where $x,z$ and $L$ are smooth functions. We wish to find $x$ (and the correspondent solution $z$ of the system) such that $z(b)$ attains an extremum.
 The necessary condition is a second-order differential equation:
 $$\frac{d}{dt}\frac{\partial L}{\partial \dot x}=\frac{\partial L}{\partial x}+\frac{\partial L}{\partial z}\frac{\partial L}{\partial \dot x},$$
for all $t\in[a,b]$.
This can be seen as an extension of the basic problem of calculus of variations. If $L$ does not depend on $z$, then integrating the differential equation along the interval $[a,b]$, we get
$$\left\{ \begin{array}{l}
\displaystyle \int_a^b \left[ L(t,x(t),\dot x(t))+\frac{z_a}{b-a} \right] \, dt \quad \to \quad \mbox{extremize} \\
x(a)=x_a, \, x(b)=x_b.
\end{array}\right.$$
As is well known, many physical phenomena are characterized by non-differentiable functions (e.g. generic trajectories of quantum mechanics \cite{Feynman}, scale-relativity without the hypothesis of space-time differentiability \cite{Nottale}). The usual procedure is to replace the classical derivative by a scale derivative, and consider the space of continuous (and non-differentiable) functions. The scale calculus of variations approach was studied in \cite{Almeida1,Cresson1,Cresson2} for a certain concept of scale derivative $\Box x(t)$:
$$\left\{ \begin{array}{l}
\displaystyle \int_a^b L(t,x(t),\Box x(t))\quad \to \quad \mbox{extremize} \\
x(a)=x_a, \, x(b)=x_b.
\end{array}\right.$$

Motivated by this problem, we define the fundamental scale variational principle of Herglotz.
First we need to define what extremum is.

\begin{Definition}
We say that $z\in C^1([a,b],\mathbb C)$ attains an extremum at $t=b$ if $z'(b)=0$.
\end{Definition}

The problem is then stated in the following way. Consider the system
\begin{equation}\label{MainProblem}\left\{ \begin{array}{l}
\dot{z}(t)=L(t,x(t),\Box x(t),z(t)), \quad \mbox{ with } \, t\in[a,b]\\
z(a)=z_a\\
x(a)=x_a, \, x(b)=x_b.
\end{array}\right.\end{equation}
For simplicity, define
$$[x,z](t):=(t,x(t),\Box x(t),z(t)).$$
We assume that
\begin{enumerate}
\item the trajectories $x$ are in $H^\alpha \cap C^1_\Box([a,b], \mathbb{R})$, $\Box x \in H^\alpha$ and the functional $z$ in $C^2([a,b],\mathbb C)$,
\item for each $x$, there exists a unique solution $z$ of the system \eqref{MainProblem}
\item $z_a,x_a,x_b$ are fixed numbers,
\item the Lagrangian $L:[a,b]\times\mathbb R\times \mathbb C^2\to\mathbb C$ is of class $C^2$.
\end{enumerate}
Observe that the solution $z(t)$ actually is a function on three variables, to know $z=z(t,x(t),\Box x(t))$. When there is no danger of mislead, we will simply write $z(t)$.
We are interested in finding a trajectory $x$ for which the corresponding solution $z$ is such that $z(b)$ attains an extremum. In particular, what necessary conditions such solutions must fulfill. These equations are called Euler-Lagrange equation types. Again, problem  \eqref{MainProblem} can be reduced to the scale variational problem in case $L$ is independent of $z$:
$$ \int_a^b L\left[(t,x(t),\Box x(t))+\frac{z_a}{b-a}\right]\,dt\quad \to \quad \mbox{extremize}.$$

\begin{Theorem}\label{TNC}If the pair $(x,z)$ is a solution of problem  \eqref{MainProblem}, and $\frac{\partial L}{\partial \Box x}[x,z]\in  H^\alpha(I,\mathbb{C})$ ($\alpha\in]0,1[$), then $(x,z)$ is a solution of the equation
\begin{equation}\label{NC}\frac{\Box}{\Box t}\left(\frac{\partial L}{\partial \Box x}[x,z](t)\right)=\frac{\partial L}{\partial x}[x,z](t)+\frac{\partial L}{\partial z}[x,z](t)\frac{\partial L}{\partial \Box x}[x,z](t),\end{equation}
for all $t\in[a,b]$.
\end{Theorem}

\begin{proof} Let $\epsilon$ be an arbitrary real, and consider variation functions of $x$ of type $x(t)+\epsilon \eta(t)$, with $\eta\in H^\beta(I,\mathbb{R}) \cap C^1_{\Box}([a,b], \mathbb{R})$ ($\beta\in]0,1[$), $\eta(a)=\eta(b)=\Box \eta(a)=0$, and
$$\lim_{h \to 0} \int_a^b\left( \frac{\Box_h}{\Box t}\left(\lambda(t)\frac{\partial L}{\partial \Box x}[x,z](t) \eta(t)\right)\right)_E \, dt =0.$$
The corresponding rate of change of $z$, caused by the change of $x$ in the direction of $\eta$, is given by
$$\theta (t)=\frac{d}{d\epsilon} \left.z(t,x(t)+\epsilon\eta(t),\Box x(t)+\epsilon \Box \eta(t))\right|_{\epsilon=0}.$$
Then
$$\begin{array}{ll}
\dot{\theta}(t)&=\displaystyle\frac{d}{dt}\frac{d}{d\epsilon} \left.z(t,x(t)+\epsilon\eta(t),\Box x(t)+\epsilon\Box \eta(t))\right|_{\epsilon=0}\\
&=\displaystyle\frac{d}{d\epsilon} \left.L(t,x(t)+\epsilon\eta(t),\Box x(t)+\epsilon\Box \eta(t),z(t,x(t)+\epsilon\eta(t),
\Box x(t)+\epsilon\Box \eta(t))\right|_{\epsilon=0}\\
&=\displaystyle\frac{\partial L}{\partial x}[x,z](t)\eta(t)+\frac{\partial L}{\partial \Box x}[x,z](t)\Box \eta(t)+\frac{\partial L}{\partial z}[x,z](t)\theta(t).
\end{array}$$
We obtain a first order linear differential equation on $\theta$, whose solution is
$$\lambda(b)\theta(b)-\theta(a)=\int_a^b\lambda(t)\left[\frac{\partial L}{\partial x}[x,z](t)\eta(t)+\frac{\partial L}{\partial \Box x}[x,z](t)\Box\eta(t)\right]dt,$$
where
$$\lambda(t)=\exp\left(-\int_a^t \frac{\partial L}{\partial z}[x,z](\tau)d\tau \right).$$
Using the fact that $\theta(a)=\theta(b)=0$, we get
$$\int_a^b\lambda(t)\left[\frac{\partial L}{\partial x}[x,z](t)\eta(t)+\frac{\partial L}{\partial \Box x}[x,z](t)\Box\eta(t)\right]dt=0.$$
Integrating by parts the second term, we obtain
$$\int_a^b\left[\lambda(t)\frac{\partial L}{\partial x}[x,z](t)-\frac{\Box}{\Box t}\left(\lambda(t)\frac{\partial L}{\partial \Box x}[x,z](t)\right)\right]\eta(t)dt+\left[\eta(t) \lambda(t)\frac{\partial L}{\partial \Box x}[x,z](t)\right]_a^b=0.$$
Since $\eta(a)=\eta(b)=0$, and $\eta$ is an arbitrary function elsewhere,
$$\lambda(t)\frac{\partial L}{\partial x}[x,z](t)-\frac{\Box}{\Box t}\left(\lambda(t)\frac{\partial L}{\partial \Box x}[x,z](t)\right)=0,$$
for all $t\in[a,b]$.
Since the function $t\mapsto \lambda(t)$ is differentiable, and the function $t\mapsto \frac{\partial L}{\partial \Box x}[x,z](t)$ is in $H^\alpha$, it follows that
$$\lambda(t)\left(\frac{\partial L}{\partial x}[x,z](t)+\frac{\partial L}{\partial z}[x,z](t)\frac{\partial L}{\partial \Box x}[x,z](t)-\frac{\Box}{\Box t}\left(\frac{\partial L}{\partial \Box x}[x,z](t)\right)\right)=0.$$
Finally, since $\lambda(t)>0$, for all $t$, we get
$$\frac{\Box}{\Box t}\left(\frac{\partial L}{\partial \Box x}[x,z](t)\right)=\frac{\partial L}{\partial x}[x,z](t)+\frac{\partial L}{\partial z}[x,z](t)\frac{\partial L}{\partial \Box x}[x,z](t),$$
for all $t \in[a,b]$.
\end{proof}

\begin{Remark} Assume that the set of state functions $x$ is $C^1([a,b],\mathbb R)$. Then equation \eqref{NC} becomes
$$\frac{d}{dt}\left(\frac{\partial L}{\partial \dot{x}}[x,z](t)\right)=\frac{\partial L}{\partial x}[x,z](t)+\frac{\partial L}{\partial z}[x,z](t)\frac{\partial L}{\partial \dot{x}}[x,z](t),$$
which is the generalized variational principle of Herglotz as in \cite{Herglotz}.
\end{Remark}

\begin{Theorem}\label{TNC2} Let the pair $(x,z)$ be a solution of the problem  \eqref{MainProblem}, but now $x(b)$ is free. Then $(x,z)$ is a solution of the equation
$$\frac{\Box}{\Box t}\left(\frac{\partial L}{\partial \Box x}[x,z](t)\right)=\frac{\partial L}{\partial x}[x,z](t)+\frac{\partial L}{\partial z}[x,z](t)\frac{\partial L}{\partial \Box x}[x,z](t),$$
for all $t\in[a,b]$, and verifies the transversality condition
$$\frac{\partial L}{\partial \Box x}[x,z](b)=0.$$
\end{Theorem}

\begin{proof} Following the proof of Theorem \ref{TNC}, the Euler-Lagrange equation is deduced. Then
$$\left[\eta(t) \lambda(t)\frac{\partial L}{\partial \Box x}[x,z](t)\right]_a^b=0.$$
Since $\eta(a)=0$ and $\eta(b)$ is arbitrary, we obtain the transversality condition.
\end{proof}

\textbf{Multi-dimensional case}

For simplicity, we considered so far one state function $x$ only, but the multi-dimensional case  $(x_1,\ldots,x_n)$ is easily studied.

\begin{Theorem}\label{TNC3} Let $\alpha\in]0,1[$ and let the vector $(x_1,\ldots,x_n,z)$ be a solution of the problem: find   $(x_1,\ldots,x_n)$ that extremizes $z(b)$, with
\begin{equation}\label{System2}\left\{ \begin{array}{l}
\dot{z}(t)=L(t,x_1(t),\ldots,x_n(t),\Box x_1(t),\ldots,\Box x_n(t),z(t)), \quad \mbox{ with } \, t\in[a,b]\\
z(a)=z_a\\
x_i(a)=x_{ia}, \, x_i(b)=x_{ib}
\end{array}\right.\end{equation}
where, for all $i\in\{1,\ldots,n\}$,
\begin{enumerate}
\item the trajectories  $x_i$ are in $H^\alpha \cap C^1_\Box([a,b], \mathbb{R})$, $\Box x_i \in H^\alpha$ and the functional $z$ in $C^2([a,b],\mathbb C)$,
\item $z_a,x_{ia},x_{ib}$ are fixed numbers,
\item $\frac{\partial L}{\partial \Box x_i}[x_1,\ldots,x_n,z]\in  H^\alpha(I,\mathbb{C})$
\item the Lagrangian $L:[a,b]\times\mathbb R^n\times \mathbb C^{n+1}\to\mathbb C$ is of class $C^2$.
\end{enumerate}
Then, for all $i\in\{1,\ldots,n\}$, $(x_1,\ldots,x_n,z)$ is a solution of the equation
$$\frac{\Box}{\Box t}\left(\frac{\partial L}{\partial \Box x_i}[x_1,\ldots,x_n,z](t)\right)=\frac{\partial L}{\partial x_i}[x_1,\ldots,x_n,z](t)+\frac{\partial L}{\partial z}[x_1,\ldots,x_n,z](t)\frac{\partial L}{\partial \Box x_i}[x_1,\ldots,x_n,z](t),$$
for all $t\in[a,b]$.
\end{Theorem}

\begin{Theorem}\label{TNC4} Let the vector $(x_1,\ldots,x_n,z)$ be a solution of the problem as stated in Theorem \ref{TNC3},  but now $x_i(b)$ is free, for all $i \in\{1,\ldots,n\}$. Then, for all $i\in\{1,\ldots,n\}$, $(x_1,\ldots,x_n,z)$ is a solution of the equation
$$\frac{\Box}{\Box t}\left(\frac{\partial L}{\partial \Box x_i}[x_1,\ldots,x_n,z](t)\right)=\frac{\partial L}{\partial x_i}[x_1,\ldots,x_n,z](t)+\frac{\partial L}{\partial z}[x_1,\ldots,x_n,z](t)\frac{\partial L}{\partial \Box x_i}[x_1,\ldots,x_n,z](t),$$
for all $t\in[a,b]$, and verifies the transversality condition
$$\frac{\partial L}{\partial \Box x_i}[x_1,\ldots,x_n,z](b)=0.$$
\end{Theorem}

\textbf{Higher-order derivatives case}

\begin{Theorem}\label{TNC5} Let $\alpha\in]0,1[$ and let the pair $(x,z)$ be a solution of the problem: find   $x$ that extremizes $z(b)$, with
$$\left\{ \begin{array}{l}
\dot{z}(t)=L(t,x,\Box x(t),\ldots,\Box^n x(t),z(t)), \quad \mbox{ with } \, t\in[a,b]\\
z(a)=z_a\\
\Box^i x(a)=x_{ia}, \, \Box^i x(b)=x_{ib}, \quad \mbox{for all } \, i\in\{0,\ldots,n-1\},
\end{array}\right.$$
where
\begin{enumerate}
\item the trajectories  $x$ are in $H^\alpha \cap C^n_\Box([a,b], \mathbb{R})$, $\Box x \in H^\alpha$ and the functional $z$ in $C^2([a,b],\mathbb C)$,
\item $z_a,x_{ia},x_{ib}$ are fixed numbers, for all $i\in\{0,\ldots,n-1\}$,
\item $\frac{\partial L}{\partial \Box^i x}[x,z]\in  H^\alpha(I^n,\mathbb{C})$, for all $i\in\{1,\ldots,n\}$,
\item $[x,z](t)=(t,x,\Box x(t),\ldots,\Box^n x(t),z(t))$ and $[x](t)=(t,x,\Box x(t),\ldots,\Box^n x(t))$,
\item the Lagrangian $L:[a,b]\times\mathbb R^{n+1}\to\mathbb R$ is of class $C^2$.
\end{enumerate}
Then, $(x,z)$ is a solution of the equation
$$\lambda(t)\frac{\partial L}{\partial x}[x,z](t)+\sum_{i=1}^n (-1)^i \frac{\Box^{i}}{\Box t^{i}}\left(\lambda(t)\frac{\partial L}{\partial \Box^ix}[x,z](t)\right)=0,$$
for all $t\in[a,b]$.
\end{Theorem}

\begin{proof} Let $x(t)+\epsilon \eta(t)$ be a variation function of $x$, with $\epsilon\in\mathbb R$ and $\eta\in H^\beta \cap  C^n_\Box([a,b], \mathbb{R})$ ($\beta\in]0,1[$).
Also, assume that the variations fulfill the conditions:
\begin{enumerate}
\item for all $i=0,\ldots,n-1$, $\Box^i \eta(a)= \Box^i \eta(b)=0$, and $\Box^n \eta(a)=0$,
\item for all $i=1, 2, \ldots, n$ and $k=0, 1,\ldots, i-1$,
$$\lim_{h \to 0} \int_a^b\left( \frac{\Box_h}{\Box t}\left(\lambda(t)\frac{\Box^k}{\Box t^k}\left(\frac{\partial L}{\partial \Box^i x}[x,z](t)\right)\Box^{i-k-1} \eta(t) \right)\right)_E \, dt =0.$$
\end{enumerate}
Define
$$\theta (t)=\frac{d}{d\epsilon} \left.z(t,x(t)+\epsilon \eta(t),\Box x(t)+\epsilon \Box \eta(t),\ldots,\Box^n x(t)+\epsilon \Box^n \eta(t))\right|_{\epsilon=0}.$$
Then
$$\dot{\theta}(t)=\frac{\partial L}{\partial x}[x,z](t)\eta(t)+\sum_{i=1}^n\frac{\partial L}{\partial \Box^ix}[x,z](t)\Box ^i\eta(t)+\frac{\partial L}{\partial z}[x,z](t)\theta(t).$$
Solving this linear ODE, we arrive at
$$\int_a^b\lambda(t)\left[\frac{\partial L}{\partial x}[x,z](t)\eta(t)+\sum_{i=1}^n\frac{\partial L}{\partial \Box^ix}[x,z](t)\Box ^i\eta(t)\right]dt=0,$$
where
$$\lambda(t)=\exp\left(-\int_a^t \frac{\partial L}{\partial z}[x,z](\tau)d\tau \right).$$
Integrating by parts $n$ times, we obtain the following:
\begin{equation*}
\begin{split}
& \displaystyle\int_a^b \left[\lambda(t)\frac{\partial L}{\partial x}[x,z](t)+\sum_{i=1}^n (-1)^i \frac{\Box^{i}}{\Box t^{i}}\left(\lambda(t)\frac{\partial L}{\partial \Box^ix}[x,z](t)\right)\right] \eta(t) dt
\\
& + \left[\sum_{i=1}^n  \sum_{k=0}^{i-1}(-1)^k \frac{\Box^{k}}{\Box t^{k}}\left(\lambda(t)\frac{\partial L}{\partial \Box^ix}[x,z](t)\right)\Box^{i-1-k} \eta(t) \right]_a^b=0,
\end{split}
\end{equation*}
and rearranging the terms, we get
\begin{equation*}
\begin{split}
& \displaystyle\int_a^b \left[\lambda(t)\frac{\partial L}{\partial x}[x,z](t)+\sum_{i=1}^n (-1)^i \frac{\Box^{i}}{\Box t^{i}}\left(\lambda(t)\frac{\partial L}{\partial \Box^ix}[x,z](t)\right)\right] \eta(t) dt
\\
& + \left[\sum_{i=1}^n  \left[\sum_{k=i}^{n}(-1)^{k-i} \frac{\Box^{k-i}}{\Box t^{k-i}}\left(\lambda(t)\frac{\partial L}{\partial \Box^kx}[x,z](t)\right)\right]\Box^{i-1} \eta(t) \right]_a^b=0.
\end{split}
\end{equation*}
Since $\Box^i \eta(a)= \Box^i \eta(b)=0$, for all $i\in\{0,\ldots,n-1\}$ and $\eta$ is arbitrary  elsewhere, we get
$$\lambda(t)\frac{\partial L}{\partial x}[x,z](t)+\sum_{i=1}^n (-1)^i \frac{\Box^{i}}{\Box t^{i}}\left(\lambda(t)\frac{\partial L}{\partial \Box^ix}[x,z](t)\right)=0,$$
for all $t\in[a,b]$.
\end{proof}

\begin{Theorem}\label{TNC6} Let the pair $(x,z)$ be a solution of the problem as stated in Theorem \ref{TNC5}, but now $\Box^i x(b)$ is free, for all $i\in\{0,\ldots,n-1\}$. Then, $(x,z)$ is a solution of the equation
$$\lambda(t)\frac{\partial L}{\partial x}[x,z](t)+\sum_{i=1}^n (-1)^i \frac{\Box^{i}}{\Box t^{i}}\left(\lambda(t)\frac{\partial L}{\partial \Box^ix}[x,z](t)\right)=0,$$
for all $t\in[a,b]$, and verifies the transversality condition
$$\sum_{k=i}^{n}(-1)^{k-i} \frac{\Box^{k-i}}{\Box t^{k-i}}\left(\lambda(t)\frac{\partial L}{\partial \Box^kx}[x,z](t)\right)=0 \quad \mbox{at} \quad t=b,$$
for all $i\in\{1,\ldots,n\}$.
\end{Theorem}

\textbf{Several independent variables case}

We generalize Theorem \ref{TNC} for several independent variables. First we fix some notations. The variable time is $t\in[a,b]$, $x=(x_1,\ldots,x_n)\in \Omega:=\prod_{i=1}^n[a_i,b_i]$ are the space coordinates and the state function is $u:=u(t,x)$.

\begin{Theorem}\label{TNC7} Let $\alpha\in]0,1[$ and let the pair $(u,z)$ be a solution of the problem: find   $u$ that extremizes $z(b)$, with
\begin{equation}\label{System4}\left\{ \begin{array}{l}
\dot{z}(t)=\displaystyle \int_\Omega L\left(t,x,u,\frac{\Box u}{\Box t},\frac{\Box u}{\Box x_1},\ldots,\frac{\Box u}{\Box x_n},z(t)\right)\, d^nx, \quad \mbox{ with } \, t\in[a,b]\\
z(a)=z_a\\
u(t,x) \quad \mbox{takes fixed values,} \quad \forall t\in[a,b]\, \forall x \in \partial\Omega\\
u(t,x) \quad \mbox{takes fixed values,} \quad \forall t\in\{a,b\}\, \forall x \in \Omega,\\
\end{array}\right.\end{equation}
where, for all $i\in\{1,\ldots,n\}$,
\begin{enumerate}
\item  the trajectories  $u$ are in $H^\alpha(I\times\Omega, \mathbb{R}) \cap C^1_\Box([a,b]\times\Omega, \mathbb{R})$, $\frac{\Box u}{\Box t},\frac{\Box u}{\Box x_i} \in H^\alpha([a,b]\times\Omega, \mathbb{C})$ and the functional $z$ in $C^2([a,b],\mathbb C)$,
\item $z_a$ is a fixed number,
\item $d^nx=dx_1\ldots dx_n$,
\item $\frac{\partial L}{\partial \Box t}[u,z],\frac{\partial L}{\partial \Box x_i}[u,z] \in  H^\alpha(I\times\Omega, \mathbb{C})$, where $\frac{\partial L}{\partial \Box t}[u,z]$ denotes the partial derivative of $L$ with respect to the variable $\frac{\Box u}{\Box t}$, and  $\frac{\partial L}{\partial \Box x_i}[u,z]$ denotes the partial derivative of $L$ with respect to the variable $\frac{\Box u}{\Box x_i}$, and $[u,z](t)=(t,x,u,\frac{\Box u}{\Box t},\frac{\Box u}{\Box x_1},\ldots,\frac{\Box u}{\Box x_n},z(t))$,
\item $L:[a,b]\times\Omega\times\mathbb R\times\mathbb C^{n+2}\to\mathbb C$ is of class $C^2$.
\end{enumerate}
Then, $(u,z)$ is a solution of the equation
$$\frac{\partial L}{\partial u}[u,z](t)+\frac{\partial L}{\partial \Box t}[u,z](t) \int_\Omega \frac{\partial L}{\partial \Box z}[u,z](t) \, d^nx-\frac{\Box}{\Box t}\left(\frac{\partial L}{\partial \Box t}[u,z](t)\right)-\sum_{i=1}^n \frac{\Box}{\Box x_i} \left(\frac{\partial L}{\partial \Box x_i}[u,z](t)\right)=0,$$
for all $t\in[a,b]$ and for all $x\in\Omega$.
\end{Theorem}

\begin{proof} Let $u(t,x)+\epsilon \eta(t,x)$ be a variation function of $u$, with $\epsilon\in\mathbb R$ and $\eta\in H^\beta(I\times\Omega, \mathbb{R}) \cap C^1_\Box([a,b]\times\Omega, \mathbb{R})$ ($\beta\in]0,1[$).
Also, assume that the variations fulfill the conditions:
\begin{enumerate}
\item $\eta(t,x)=0,\quad \forall t\in[a,b]\,\forall x \in \partial\Omega$,
\item $\eta(t,x)=0,\quad \forall t\in\{a,b\}\, \forall x \in \Omega$,
\item $\frac{\Box \eta}{\Box t}(a,x)=\frac{\Box \eta}{\Box x_i}(a,x)=0,\quad \forall x \in \Omega$,
\item for all $i=1, 2, \ldots, n$,
$$\lim_{h \to 0} \int_a^b\left( \frac{\Box_h}{\Box t}\left(\lambda(t)\frac{\partial L}{\partial \Box t}[u,z](t) \eta(t)\right)\right)_E \, dt =0.$$
and
$$\lim_{h \to 0} \int_a^b\left( \frac{\Box_h}{\Box x_i}\left(\lambda(t)\frac{\partial L}{\partial \Box x_i}[u,z](t) \eta(t)\right)\right)_E \, dt =0,$$
where
$$\lambda(t)=\exp\left(-\int_a^t\int_\Omega \frac{\partial L}{\partial z}[u,z](\tau) \, d^nx \, d\tau \right).$$
\end{enumerate}
Let
$$\theta (t)=\frac{d}{d\epsilon} \left.z\left(t,x,u+\epsilon\eta,\frac{\Box u}{\Box t}+\epsilon \frac{\Box \eta}{\Box t},\frac{\Box u}{\Box x_1}+\epsilon \frac{\Box \eta}{\Box x_1},\ldots,\frac{\Box u}{\Box x_n}+\epsilon \frac{\Box \eta}{\Box x_n}\right)\right|_{\epsilon=0}.$$
Proceeding with some calculations, we arrive at the ODE
$$\dot{\theta}(t)-\int_\Omega \frac{\partial L}{\partial z}[u,z](t) \, d^nx \,\, \theta(t) =\int_\Omega \frac{\partial L}{\partial u}[u,z](t)\eta+\frac{\partial L}{\partial \Box t}[u,z](t)\frac{\Box\eta}{\Box t}+\sum_{i=1}^n\frac{\partial L}{\partial \Box x_i}[u,z](t)\frac{\Box \eta}{\Box x_i}\, d^nx.$$
Solving the ODE, and taking into consideration that $\theta(a)=\theta(b)=0$, we get
$$\int_a^b \int_\Omega \lambda(t)\left[\frac{\partial L}{\partial u}[u,z](t)\eta+\frac{\partial L}{\partial \Box t}[u,z](t)\frac{\Box\eta}{\Box t}+\sum_{i=1}^n\frac{\partial L}{\partial \Box x_i}[u,z](t)\frac{\Box \eta}{\Box x_i}\right]\, d^nx  \, dt=0.$$
Integrating by parts, and considering the boundary conditions over $\eta$, we get
$$\int_a^b\int_\Omega \left[\lambda(t)\frac{\partial L}{\partial u}[u,z](t)
-\frac{\Box}{\Box t}\left(\lambda(t)\frac{\partial L}{\partial \Box t}[u,z](t)\right)-\sum_{i=1}^n\frac{\Box}{\Box x_i} \left(\lambda(t)\frac{\partial L}{\partial x_i}[u,z](t)\right)\right]\eta d^nxdt=0.$$
By the arbitrariness of $\eta$, it follows that for all $t\in[a,b]$ and for all $x\in\Omega$,
$$\lambda(t)\frac{\partial L}{\partial u}[u,z](t)
-\frac{\Box}{\Box t}\left(\lambda(t)\frac{\partial L}{\partial \Box t}[u,z](t)\right)-\sum_{i=1}^n\frac{\Box}{\Box x_i} \left(\lambda(t)\frac{\partial L}{\partial x_i}[u,z](t)\right)=0.$$
Since $\lambda(t)>0$, this condition implies that
$$\frac{\partial L}{\partial u}[u,z](t)+\frac{\partial L}{\partial \Box t}[u,z](t) \int_\Omega \frac{\partial L}{\partial \Box z}[u,z](t) \, d^nx-\frac{\Box}{\Box t}\left(\frac{\partial L}{\partial \Box t}[u,z](t)\right)-\sum_{i=1}^n \frac{\Box}{\Box x_i} \left(\frac{\partial L}{\partial \Box x_i}[u,z](t)\right)=0,$$
for all $t\in[a,b]$ and for all $x\in\Omega$, and the theorem is proved.
\end{proof}


\section*{Acknowledgements}

We would like to thank the two reviewers for their insightful comments on the paper, as these comments led us to an improvement of the work.
This work was supported by Portuguese funds through the CIDMA - Center for Research and Development in Mathematics and Applications,
and the Portuguese Foundation for Science and Technology (FCT-Funda\c{c}\~ao para a Ci\^encia e a Tecnologia), within project UID/MAT/04106/2013.




\begin{thebibliography}{99}

\bibitem{Almeida1}
R. Almeida and D.F.M. Torres. Nondifferentiable variational principles in terms of a quantum operator.  Math. Methods Appl. Sci.  {\bf 34} (2011) 2231-–2241.

\bibitem{Cresson1}
J. Cresson and I. Greff. A non-differentiable Noether's theorem.  J. Math. Phys.  {\bf 52} (2011) No.2, 023513, 10 pp.

\bibitem{Cresson2}
J. Cresson  and I. Greff. Non-differentiable embedding of Lagrangian systems and partial differential equations. J. Math. Anal. Appl.  {\bf 384} (2011) No. 2, 626-–646.

\bibitem{Feynman}
R. Feynman and A. Hibbs. Quantum mechanics and path integrals, MacGraw-Hill,1965.

\bibitem{Georgieva1}
B. Georgieva and R. Guenther. First Noether-type theorem for the generalized variational principle of Herglotz. Topol. Methods Nonlinear Anal. {\bf 20} (2002) No. 2, 261--273.


\bibitem{Georgieva2}
B. Georgieva and R. Guenther. Second Noether-type theorem for the generalized variational principle of Herglotz. Topol. Methods Nonlinear Anal. {\bf 26} (2005) No. 2, 307--314.

\bibitem{Georgieva3}
B. Georgieva, R. Guenther and T. Bodurov. Generalized variational principle of Herglotz for several independent variables. First Noether-type theorem, J. Math. Phys. {\bf 44} (2003) No. 9, 3911--3927.

\bibitem{Guenther1}
R.B. Guenther, J.A. Gottsch and D.B. Kramer, The Herglotz algorithm for constructing canonical transformations, SIAM Rev. {\bf 38} (1996) No. 2, 287--293.

\bibitem{Guenther2}
R.B. Guenther, C.M. Guenther and J.A. Gottsch, The Herglotz Lectures on Contact Transformations and Hamiltonian Systems, Lecture Notes in Nonlinear Analysis, Vol. 1, Juliusz Schauder Center for Nonlinear Studies, Nicholas Copernicus University, Tor\'{u}n, 1996.

\bibitem{Herglotz}
G. Herglotz, Ber\"{u}hrungstransformationen, Lectures at the University of G\"{o}ttingen, G\"{o}ttingen, 1930.

\bibitem{Nottale}
L. Nottale. The theory of scale relativity. Internat. J. Modern Phys. A {\bf 7} (1992) No. 20, 4899-–4936.

\bibitem{Orum}
J.C. Orum, R.T. Hudspeth, W. Black and R.B. Guenther, Extension of the Herglotz algorithm to nonautonomous canonical transformations, SIAM Rev. {\bf 42} (2000) No. 1, 83--90.

\bibitem{santos}
S. Santos, N. Martins, D.F.M. Torres, Higher-order variational problems of Herglotz type, Vietnam Journal of Mathematics {\bf 42} (2014) No. 4, 409--419.

\end{thebibliography}
\end{document}